\documentclass[a4paper]{article}
\usepackage[12pt]{extsizes} 
\usepackage[utf8]{inputenc}
\usepackage[english, russian]{babel}
\usepackage{setspace,amsmath}
\usepackage{amsthm}
\usepackage{amssymb}
\usepackage[left=20mm, top=15mm, right=15mm, bottom=15mm, nohead, footskip=10mm]{geometry} 

\theoremstyle{plain}
\newtheorem{thm}{Theorem}[section]
\newtheorem{lm}{Lemma}[section]
\newtheorem{st}{Statement}
\newtheorem{prop}{Proposition}

\theoremstyle{definition}
\newtheorem*{defn}{Def}

\newtheorem*{cor}{Corollary}

\theoremstyle{remark}
\newtheorem*{rem}{Rem}

\newcommand{\Tr}{\mathop{\mathrm{Tr}}\nolimits}

\begin{document} 
\selectlanguage{english}
	\begin{center}
		\textbf{Alternative proof of upper bound for spanning trees in a graph}

		\textbf{K. V. CHELPANOV}
	\end{center}

\begin{center}
\begin{footnotesize}
ABSTRACT. We give a proof of upper bound of spanning trees in a graph in terms of vertex degrees using linear algebra techniques and generalize it to the case of multigraphs. Also we obtain for which graphs the inequality turns into equality.
\end{footnotesize}
\end{center}

\section{Introduction}

Let $ G = (V,E) $ denote the undirected connected (multi)graph without loops and $ \tau(G) $ denote number of spanning trees in $ G $. In this paper we want to find upper bound of $ \tau(G) $.

A common approach for counting spanning trees is \textit{Kirchhoff's Matrix Tree Theorem} and its corollary that $ \tau(G) $ can be expressed in terms of Laplacian eigenvalues. There are many various upper bounds for $\tau(G)$ in terms of number of vertices, number of edges and degrees of vertices. 
In \cite{narayanan} the following upper bound for number of spanning trees in a graph was proved by induction using stronger result about multigraphs:
\begin{thm}
	Let $ G $ be a simple graph, $ d_1 \leqslant \dots \leqslant d_n $ --- degrees of its vertices. Then
	$$
		\tau(G) \leqslant \frac{(1+d_1)\dots(1+d_n)}{n^2}.
	$$
\end{thm}
 This paper is organised as follows. We give an alternative proof of this inequality using linear algebraic techniques. After that we formulate the analogous result for multigraphs and prove its sharpness for complete multigraphs.

\section{Preliminary lemmas}

\begin{defn}
	Let $ x = (x_1 \geqslant \dots \geqslant x_n) $ and $ y = (y_1 \geqslant \dots \geqslant y_n) $ be two finite seqences of real numbers. We say that $ x $ \textbf{majorizes} $ y $ ($x \succ y$) 		iff for every $ k \in [1..n] $
	$$
		x_1 + \dots + x_k \geqslant y_1 + \dots + y_k
	$$ 
	and $ x_1 + \dots + x_n = y_1 + \dots + y_n $.
\end{defn} 

\begin{lm}[Karamata's inequality, \cite{kar}]
	Let $ f \colon I \to \mathbb{R} $ be a real-valued convex function, $ I $ is an interval on the real line, $ x = (x_1 \geqslant \dots \geqslant x_n) $ and $ y = (y_1 \geqslant \dots 			\geqslant y_n) $ are two finite sequences of numbers in I that $ x \succ y $. Then the following inequality is true:
	$$
		f(x_1) + \dots + f(x_n) \geqslant f(y_1) + \dots + f(y_n).
	$$ 
	If $ f $ is strictly convex, then the equality holds iff the sequences coincide, i.e. $ x_i = y_i $ for every $ i \in [1..n] $. 
\end{lm}

\begin{cor}
	 If $ x = (x_1, \dots, x_n) $, $ y = (y_1, \dots, y_n) $ are two sequences of positive numbers and $ x \succ y $, then $ x_1 \cdot \ldots \cdot x_n \leqslant y_1 \cdot \ldots \cdot y_n $.
\end{cor}

\begin{proof}
	Function $ f(x) = -\log x $ is defined on $ \mathbb{R}_+ $ and it is convex. Applying Karamata's inequality to this function and given sequences we obtain
	$$
		-\log x_1 - \dots -\log x_n \geqslant -\log y_1 -\dots -\log y_n.
	$$
	Reversing the sign and exponentiating with base $ e $, we obtain the desired inequality.
\end{proof}

\begin{lm}[Schur's inequality\cite{brou}]
	Let $ A $ be the symmetric real-valued $ n \times n $ matrix with diagonal elements $ d_1 \geqslant \dots \geqslant d_n $ and eigenvalues $ \lambda_1 \geqslant \dots \geqslant 			\lambda_n $. Then $ d_1 + \dots + d_k \leqslant \lambda_1 + \dots + \lambda_k $ for every $ k \in [1..n] $. In other words, $ \lambda \succ d $.
\end{lm}

\section{Basic properties about Laplacian eigenvalues}
\begin{defn}
	Let $ G $ be a simple graph (without multiple edges and loops). \textbf{Laplacian} $ L(G) $ of this graph is the following matrix:
	\begin{equation*}
		L_{i,j} = 
 		\begin{cases}
   			deg(v_i) &\text{if $ i = j $}\\
   			-1 &\text{if $ i \ne j $ and $ v_iv_j \in E(G)$}\\
			0 &\text{if $ i \ne j $ and $ v_iv_j \notin E(G)$}
 		\end{cases}
	\end{equation*}
\end{defn}

\begin{prop}
	Let $ 0 = \mu_1 \leqslant \mu_2 \leqslant \ldots \leqslant \mu_n $ be the spectrum of $ L(G) $. Then for every $ k\in [1..n] $ $ \mu_k \in [0,n] $.
\end{prop}

\begin{prop}
	Let $ 0 = \mu_1 \leqslant \mu_2 \leqslant \ldots \leqslant \mu_n $ be the spectrum of $ L(G) $. Define \textbf{the complement} of G: for every $ v_i, v_j \in V(G) $:
	\begin{align*}
		v_iv_j \in E(G) \iff v_iv_j \notin E(\overline{G});
		\\
		v_iv_j \notin E(G) \iff v_iv_j \in E(\overline{G}).
	\end{align*}
	Then the spectrum of $ L(\overline{G}) $ is $ 0 \leqslant n-\mu_n \leqslant \dots \leqslant n-\mu_2 $.
\end{prop}

\begin{thm}[Kirchhoff's theorem, \cite{kir}]
	Number of spanning trees in graph $ G $ (denote by $ \tau(G) $) is equal to every cofactor of $ L(G) $.
\end{thm}

\begin{cor}[\cite{cvet}]
	If $ 0 = \mu_1 \leqslant \mu_2 \leqslant \dots \leqslant \mu_n $ is the spectrum of $ L(G) $, then $ \tau(G) = \frac{\mu_2\cdot\ldots\cdot\mu_n}{n} $.
\end{cor}

Now proceed to give the proof of main theorem.

\section{Proof of the main theorem}
Let us remind the main theorem.
\begin{thm}
	Let $ G $ be a simple graph, $ v(G) = n $, $ d_1 \leqslant \dots \leqslant d_n $ be its degree sequence. Then the following equality holds:
	$$
		\tau(G) \leqslant \frac{(1+d_1)\cdot\ldots\cdot(1+d_n)}{n^2}.
	$$
\end{thm}
\begin{proof}
From corollary of Kirchhoff's theorem we know that if $ 0 =\mu_1 \leqslant \mu_2 \leqslant \dots \leqslant \mu_n $ are eigenvalues of $ L(G) $, then $ \tau(G) = \frac{\mu_2 \cdot \ldots \cdot \mu_n}{n} $. Using this equality, we obtain
$$
	\frac{\mu_2 \cdot \ldots \cdot \mu_n}{n} \leqslant \frac{(1+d_1)\dots(1+d_n)}{n^2}.
$$
Hence, we should prove the following inequality:
$$
	n\cdot\mu_n \cdot \ldots \cdot \mu_2 \leqslant (1+d_n)\cdot \ldots \cdot(1+d_1).
$$ 

Notice that sums of multipliers in both parts are equal. Indeed,
$$
	n + \mu_n + \dots + \mu_2 = n + \Tr(L(G)) = n + d_1 + \dots + d_n = (d_n+1) + \dots + (d_1+1).
$$

Now consider graph $ \overline{G} $ --- the complement of $ G $. Its degree sequence is $ n-1-d_n \leqslant \dots \leqslant n-1-d_1 $ and its spectrum is $ 0 \leqslant n-\mu_n \leqslant \dots \leqslant n-\mu_2 $. Applying Schur's inequality we obtain
$$
	(n-1-d_1) + \dots + (n-1-d_{k-1}) \leqslant (n-\mu_2) + \dots + (n-\mu_k),
$$
which is equivalent to
$$
	\mu_2 + \dots + \mu_k \leqslant d_1 + \dots + d_{k-1} + (k-1) = (1 + d_1) + \dots + (1 + d_{k-1}). 
$$
So,
$$
	n + \mu_n + \dots + \mu_{k+1} \geqslant (1+d_n) + \dots + (1+d_k)
$$
Thus, we get that sequence $ (n, \mu_n, \dots, \mu_2) $ majorizes $ (1+d_n, \dots, 1+d_1) $. Now apply the corollary of Karamata's inequality for products and get the desired inequality.
\end{proof}

\subsection{Generalization for multigraphs} 

As a corollary, we give the variant of the inequality of the main theorem in case of multigraphs. 

\begin{defn}
	Define the \textbf{complement} of multigraph $ G $ with maximal edge multiplicity $ \Delta $ ($ \Delta = \max_{i\ne j} |L_{i,j}(G)| $): for every edge $ e \in E(G) $ 
	$$
		\mu_G(e) = k \iff \mu_{\overline{G}}(e) = \Delta - k,
	$$
	where $ \mu_G(e) $ is multiplicity of $ e $ in $ G $.
\end{defn}

\begin{prop}
	Let $ G $ be a multigraph with maximal edge multiplicity $ \Delta $, $ 0 = \mu_1 \leqslant \mu_2 \leqslant \dots \leqslant \mu_n $ --- its spectrum. Then for every $ k\in [1..n] $ $ \mu_n 		\in [0,n\Delta] $.
\end{prop}

\begin{prop}
	Let $ G $ be a multigraph with maximal edge multiplicity $ \Delta $, $ 0 = \mu_1 \leqslant \mu_2 \leqslant \dots \leqslant \mu_n $ --- its spectrum. Then the spectrum of $ L(\overline{G}) $ 	is $ 0 \leqslant n\Delta - \mu_n \leqslant \dots \leqslant n\Delta - \mu_2 $.
\end{prop}

\begin{thm}
	Let $ G $ be a multigraph, $ v(G) = n $, $ \Delta $ --- its maximal edge multiplicity  and $ d_1 \leqslant \dots \leqslant d_n $ --- its degree sequence. Then the following inequality is true:
	$$
		\tau(G) \leqslant \frac{(\Delta+d_1) \dots (\Delta+d_n)}{\Delta n^2}	
	$$
\end{thm}

\begin{proof}
	We will apply corollary of Karamata's inequality for sequences $ (n\Delta, \mu_n, \dots, \mu_2) $ and $ (\Delta+d_n, \dots, \Delta+d_1) $ in the similar fashion as in proof of the main 		theorem. Considering the multigraph $ \overline{G} $ and using proposition about its eigenvalues, we obtain the following inequality for every $ k\in[1..n]$:
	$$
		((n-1)\Delta-d_1) + \dots + ((n-1)\Delta-d_{k-1}) \leqslant (n\Delta - \mu_2) + \dots + (n\Delta - \mu_k),
	$$
	which is equivalent to
	$$
		\mu_2 + \dots + \mu_k \leqslant (d_1 + \Delta) + \dots (d_{k-1} + \Delta).
	$$
	So,
	$$
		n\Delta + \mu_n + \dots + \mu_{k+1} \geqslant (d_n + \Delta) + \dots + (d_k + \Delta).
	$$
	We get that $ (n\Delta, \mu_n, \dots, \mu_2) \succ (d_n+\Delta, \dots, d_1+\Delta) $, as we needed. Application of Karamata's inequality for products completes the proof.
\end{proof}

Now consider in what graphs the inequality turns into equality.

\begin{st}
	The inequality turns into equality iff $ G $ is a complete multigraph with all edge multiplicities equal to $ \Delta $.
\end{st}

\begin{proof}
	We have that $ n\Delta \cdot \mu_n \cdot \ldots \cdot \mu_2 = (d_n+\Delta)\dots(d_1+\Delta) $. It is true iff the sequences coincide, i.e. $ n\Delta = d_n + \Delta $, $ \mu_n = d_{n-1}		+ \Delta $, $ \dots $, $ \mu_2 = d_1 + \Delta $. From the first equality we obtain that $ d_n = (n-1)\Delta $, and since every edge of $ G $ has multiplicity no greater than $ \Delta $, we 
	get that every edge incident to the vertex with degree $ d_n $ has multiplicity $ \Delta $. Notice that $ \overline{G} $ has at least two connectivity components, thus its second 			eigenvalue is zero, $ \mu_n = n\Delta $ and $ d_{n-1} = (n-1)\Delta $. So, every edge incident to the vertex with degree $ d_{n-1} $ has multiplicity $ \Delta $. Iterating this process, 		we obtain that all non-zero Laplacian eigenvalues of $ G $ are equal to $ n\Delta $ and every vertex has degree $ (n-1)\Delta $. Thus, every edge has multiplicity $ \Delta $. 
\end{proof}

\newpage
\section{Applications}
\subsection{Product of graphs}

\begin{defn}
	Let $ G_1, \ldots, G_n $ be $ n $ graphs. \textbf{The product of $G_i$'s}, which is denoted by $ G_1\nabla G_2 \nabla \ldots \nabla G_n $, is defined in the following way:
	\begin{align*}
		G = G_1\nabla G_2 \nabla \ldots \nabla G_n, V(G) = \bigsqcup_{i\in[1..n]}V(G_i),
		\\
		E(G) = \bigcup_{i\in[1..n]}E(G_i) \cup \{v_iv_j\}_{1\leqslant i < j \leqslant n, v_i \in V(G_i), v_j \in V(G_j)}
	\end{align*}
\end{defn}

Suppose that $ v(G_i) = v_i $ for every $ i\in [1..n] $, $ V(G) = v $ and $ 0 = \mu^i_1 \leqslant \ldots \leqslant \mu^i_{v_i} $ is spectrum of $ L(G_i) $. Then it is known that
$$
	\tau(G) = v^{n-2}\prod\limits_{i\in[1..n]}\prod\limits_{j\in[2..v_i]}(v - v_i + \mu^i_j)
$$

And since for every $ i\in [1..n] $ every product can be estimated by $ \prod\limits_{j\in[2..v_i]}(v - v_i + \mu^i_j) \leqslant \frac{\prod\limits_{j\in[1..v_i]}(v - v_i + d^i_j)}{(v-v_i)}$. Thus, we have
\begin{thm}
	Let $ G_1 $, $ \ldots $, $ G_n $ be $ n $ graphs, $ v(G_i) = v_i $; $ \sum\limits_{i=1}^{n} v_i = v $, $ G = G_1\nabla\ldots\nabla G_n $. Then the following inequality holds:
	$$
		\tau(G) \leqslant v^{n-2}\frac{\prod\limits_{u\in V(G)}d_G(u)}{\prod\limits_{i\in[1..n]}(v-v_i)}
	$$ 
\end{thm}

\begin{rem}
	The equality holds iff $ G_i = \overline{K_{v_i}} $, and hence $ G = K_{v_1,\ldots,v_n} $.
\end{rem}

\subsection{Cartesian product of graphs}

\begin{defn}
	Let $ G $ and $ H $ be two graphs. \textbf{The Cartesian product of $ G $ and $ H $}, which is denoted by $ G \times H $, is defined in the following way:

	if $ V(G) = \{u_1, \ldots, u_m\} $ and $ V(H) = \{v_1, \ldots, v_n\} $. Then $ V(G\times H) = \{w_{ij}\}_{i\in[1..m], j\in[1..n]} $, $ E(G\times H) = \{w_{ij}w_{ik}\}_{i\in[1..m],1\leqslant j<k \leqslant n, v_jv_k\in E(H)} \cup \{w_{ik}w_{jk}\}_{k\in[1..n], 1\leqslant i<j \leqslant m, u_iu_j \in E(G)} $.
\end{defn}

\begin{lm}
	Let $ G $ and $ H $ be two graphs, $ 0 = \lambda_1 \leqslant \ldots \leqslant \lambda_m $ is spectrum of $ L(G) $ and $ 0 = \mu_1 \leqslant \ldots \leqslant \mu_n $ is spectrum of $ L(H) 			$. Then spectrum of $ L(G\times H) $ is $ \{\lambda_i + \mu_j\}_{i\in[1..m], j\in[1..n]} $. 
\end{lm}

Using previous notions, we can obtain the following formula for number of spanning trees in $ G\times H $:
$$
	\tau(G\times H) = \frac{\lambda_2\cdot\ldots\cdot \lambda_m}{m}\cdot\frac{\mu_2\cdot\ldots\cdot \mu_n}{n}\cdot\prod\limits_{i\in[2..m], j\in[2..n]}(\lambda_i + \mu_j)
$$
Considering that first and second multipliers in the RHS are equal to $ \tau(G) $ and $ \tau(H) $, we get
$$
	\tau(G\times H) = \tau(G)\cdot\tau(H)\cdot\prod\limits_{i\in[2..m], j\in[2..n]}(\lambda_i + \mu_j)
$$
Fix $ i\in [2..m] $ Then, if $ \lambda_i > 0 $, we can use Karamata nad Schur's inequalities and obtain
$$
	(\lambda_i + \mu_2)\cdot\ldots\cdot(\lambda_i + \mu_n) \leqslant \frac{\prod\limits_{j\in[1..n]}(\lambda_i + d_H(v_j))}{\lambda_i}
$$
If all such $ \lambda_i $ are strictly positive, then $ G $ is connected and, multiplying these inequalities by all $ i $'s in this interval, we obtain the following:
$$
	\tau(G\times H) = \tau(G)\cdot\tau(H)\cdot  \frac{\prod\limits_{i\in[2..m]}\prod\limits_{j\in[1..n]}(\lambda_i + d_H(v_j))}{\lambda_2\cdot\ldots\cdot\lambda_m} = \frac{\tau(H)}{m}\prod\limits_{j\in[1..n]}\prod\limits_{i\in[2..m]}(\lambda_i + d_H(v_j))
$$
Now fix $ j\in [1..n] $. Then, if $ d_H(V_j) \ne 0 $
$$
	\prod\limits_{i\in[2..m]}(\lambda_i + d_H(v_j)) \leqslant \frac{\prod\limits_{k\in[1..m]}(d_G(u_k) + d_H(v_j))}{d_H(v_j)}
$$
If there is no isolated vertices in $ H $, then, multiplying these inequalities by all $ j $'s in this interval, we obtain the following:
$$
	\prod\limits_{j\in[1..n]}\prod\limits_{i\in[2..m]}(\lambda_i + d_H(v_j)) \leqslant \frac{\prod\limits_{i\in[1..m], j\in[1..n]}(d_G(u_i) + d_H(v_j))}{\prod\limits_{j\in[1..n]}d_H(v_j)}
$$
hence,
$$
	\tau(G\times H) \leqslant \tau(H)\cdot\frac{\prod\limits_{i\in[1..m], j\in[1..n]}(d_G(u_i) + d_H(v_j))}{m\cdot\prod\limits_{j\in[1..n]}d_H(v_j)}
$$

Now formulate this result as a theorem.
\begin{thm}
	Let $ G $ and $ H $ be two graphs, $ v(G) = m $, $ v(H) = n $, $ G $ is connected and $ H $ has no isolated vertices. Then the following inequality holds:
	$$
		\tau(G\times H) \leqslant \tau(H)\cdot\frac{\prod\limits_{u\in V(G), v\in V(H)}(d_G(u) + d_H(v))}{m\cdot\prod\limits_{v\in V(H)}d_H(v)}
	$$
\end{thm}

\begin{center}
\begin{footnotesize}
ACKNOWLEDGEMENTS
\end{footnotesize}
\end{center}

The author is grateful to Fedor Petrov and Alexander Polyanski for suggestions and discussions concerned with this paper and their support. The main results are formulated in author's bachelor thesis.

The work is supported by Ministry of Science and Higher Education of the Russian Federation, agreement №075-15-2022-289.

\begin{flushleft}
\begin{footnotesize}
	Leonard Euler Internathional Mathematical Institute in St. Petersburg,

	Line 14th (Vasilyevsky Island), 29B,
	
	St. Petersburg 199178, Russia

	\textit{E-mail:} \tt{chelpkostya@yandex.ru} 
\end{footnotesize}
\end{flushleft}
\end{document}